\documentclass{article}[12pt]
\usepackage{url, graphicx, xcolor, hyperref, geometry}
\usepackage{latexsym,amsbsy}
\usepackage{lipsum}
\usepackage{amsfonts}
\usepackage{graphicx}
\usepackage{epstopdf}
\usepackage{algorithmic}
\usepackage{tikz}
\usetikzlibrary{matrix}
\usetikzlibrary{shapes,patterns,arrows,snakes,decorations.shapes}
\usetikzlibrary{positioning,fit,calc}
\usetikzlibrary{plotmarks}
\usepackage{amssymb,amsmath,amsthm}
\usepackage{smartdiagram}
\usepackage{mathabx}
\usepackage{float}
\usepackage{enumerate}
\usepackage{mnsymbol}

\newtheorem{theorem}{Theorem}[section]

\newtheorem{corollary}{Corollary}[section]
\newtheorem{remark}{Remark}[section]

\newtheorem{definition}{Definition}[section]

\usepackage[font=small,labelfont=bf, labelsep=space]{caption}
\hypersetup{
    colorlinks,
    linkcolor={blue},
    citecolor={blue},
    urlcolor={blue}
}

\begin{document}
%------------------------------------------------------------------------------
 
\title{Optimal smoothing factor with coarsening by  three for the MAC scheme for the Stokes equations}

\author{Yunhui He\thanks{Department of Computer Science, The University of British Columbia, Vancouver, BC, V6T 1Z4, Canada,   \tt{yunhui.he@ubc.ca}.}  }

\maketitle

\begin{abstract}
 In this work, we propose a local Fourier analysis for multigrid methods with coarsening by a factor of three for the staggered finite-difference method applied to the Stokes equations.   In \cite{YH2021massStokes},   local Fourier analysis has been applied to  a  mass-based Braess-Sarazin relaxation,  a mass-based $\sigma$-Uzawa relaxation,   and  a mass-based distributive relaxation,  with standard coarsening on staggered grids for the Stokes equations.  Here, we consider multigrid methods with coarsening by three  for these relaxation schemes.  We derive theoretically optimal smoothing factors for this coarsening strategy. The optimal smoothing factors  of coarsening by three are nearly equal to those obtained from standard coarsening. Thus, coarsening by three is superior computationally.   Moreover,  coarsening by three  generates a nested hierarchy of grids, which simplifies and unifies the construction of grid-transfer operators.
\end{abstract}

 
 \vspace{2mm}

%===============================================================================
\section{Introduction}
\label{sec:intro}
 
We are interested in multigrid methods \cite{MR1156079}  for the numerical solution of the Stokes equations. In the literature,  different types of discretizations have been applied to the Stokes  equations, for example,  finite element methods \cite{elman2006finite,dohrmann2004stabilized}, finite difference methods \cite{han1998new,da2009mimetic},  and finite volume methods \cite{ye2006discontinuous,droniou2009study}.   As is well-known,  the marker and cell method (MAC), a finite difference discretization,  is one of the most effective numerical schemes for solving  the Stokes equations \cite{oosterlee2006multigrid,he2018local} and in this paper we focus on this method. It was introduced  in \cite{lebedev1964difference, welch1965mac} for solving viscous, incompressible and transient fluid-flow problems, and has been extended to other problems.   Many studies of multigrid methods focus on standard coarsening for the MAC scheme for the Stokes equations.   The unknowns (velocity and pressure) are located at different locations: the velocity components are placed at the cell faces and the pressure  is placed in the center of the cells,  and, as a result, for standard coarsening, the locations of unknowns for coarse-grid problems are not  subsets of these on the fine grid.

The choice of multigrid components, such as relaxation schemes and  grid-transfer operators,  plays an important role in designing fast algorithms.    Block-structured relaxation schemes are often used for the Stokes equations, such as  Braess-Sarazin relaxation \cite{braess1997efficient,YH2021massStokes},  Uzawa relaxation \cite{MR833993,MR3217219,MR1302679,drzisga2018analysis}, and  distributive relaxation \cite{bacuta2011new,oosterlee2006multigrid,MR3071182,chen2015multigrid,MR1049395}.   
In our recent work \cite{YH2021massStokes},   we presented three  block-structured relaxation schemes for solving the Stokes equations discretized by MAC  method: a mass-based  Braess-Sarazin relaxation ($Q$-BSR),  a mass-based $\sigma$-Uzawa relaxation ($Q$-$\sigma$-Uzawa), and a mass-based distributive relaxation ($Q$-DR), where  we used a mass matrix $Q$ derived from bilinear finite elements in two dimensions to approximate the inverse of a scalar Laplacian discretized by five-point finite difference method.  Local Fourier analysis (LFA) was used to study  smoothing with standard coarsening.  We obtained  an optimal smoothing factor of  $\frac{1}{3}$  for the mass-based  distributive  and   Braess-Sarazin  relaxation schemes, and $\sqrt{\frac{1}{3}}$ for the mass-based $\sigma$-Uzawa relaxation.  These relaxation schemes show high efficiency.

However, for standard coarsening in MAC scheme,  different types of grid-transfer operators are needed  for velocity components and pressure \cite{MR1049395}.  In contrast,   coarsening by three  generates a nested hierarchy of grids, which simplifies the definition of grid-transfer operators. A coarsening-by-three strategy has the potential advantage of coarsening more quickly and reducing the number of levels.  Multigrid methods for the generalized Stokes equations with coarsening by  three  with distributive Gauss-Seidel smoothing is presented in \cite{butt2018multigrid}. For multigrid methods with coarsening by a factor of three applied to other problems, see \cite{dendy2010black,gaspar2009geometric,yavneh2012nonsymmetric}.  Unfortunately, there are few  studies on coarsening by three for the Stokes equations.  Thus, we study $3h$-coarsening for the MAC scheme.

Motivated by the advantages of multigrid with coarsening by three and the high efficiency of mass-based relaxation schemes, we wish to explore what  the optimal smoothing factors  are for these relaxation schemes.  Choosing appropriate algorithmic parameters is challenging. Thus, we apply LFA to help us identify proper parameters and quantitatively predict multigrid convergence speed.  

The main contribution of this work is  the presentation of a theoretical analysis of optimal smoothing factors of  three mass-based   multigrid relaxation schemes for staggered grids using a three-coarsening strategy for the Stokes equations. We derive an optimal smoothing factor and show that it is  $\frac{17}{47}\approx 0.362$ for  $Q$-BSR and $Q$-DR, and  an optimal smoothing  factor of $\sqrt{\frac{17}{47}}\approx 0.601$ for  $Q$-$\sigma$-Uzawa relaxation.  Thus, $Q$-BSR outperforms other two relaxation schemes. Note that the optimal smoothing factors for coarsening  by three are very close to those ($\frac{1}{3}$ and $\sqrt{\frac{1}{3}}\approx 0.577$) for standard coarsening.  It means that the computation work of coarsening is competitive with standard coarsening.    

To avoid solving Schur complement system exactly in $Q$-BSR,  we propose an  inexact version of $Q$-BSR, called $Q$-IBSR,  where  one sweep of  weighted-Jacobi iteration is applied to the Schur complement system.  Numerically, we find that $Q$-IBSR achieves the same convergence factor as that of exact version, that is,   $\frac{17}{47}$. Moreover, we study the influence of different types of grid-transfer operators on the actual multigrid convergence. We test the Stokes equations with Dirichlet boundary conditions, and numerical results show that $Q$-IBSR is not too sensitive to the boundary conditions, but  we see degradation on convergence for $Q$-DR $Q$-$\sigma$-Uzawa relaxation. 
  
The reminder of the paper is organized as follows.   In Section \ref{sec:discretization}, we review  staggered finite-difference discretization  for the Stokes equations.  In Section \ref{sec:smoothing-analysis}, we derive  optimal smoothing factors of LFA for  three mass-based block-structured multigrid relaxation schemes proposed in \cite{YH2021massStokes} with coarsening by three.   In Section \ref{sec:Numer}, we study different grid-transfer operators by LFA and present some numerical results  to validate our theoretical results.  Some conclusions  are drawn in Section \ref{sec:concl}.
\section{Discretization}\label{sec:discretization}

%==================================================================================================
Consider  the following Stokes equations in two dimensions
\begin{eqnarray}
  -\Delta\boldsymbol{u}+\nabla p&=&\boldsymbol{f},\label{eq:Stokes1}\\
  \nabla\cdot \boldsymbol{u}&=&0,\label{eq:Stokes2}
\end{eqnarray}
where $\boldsymbol{u}=\begin{pmatrix} u \\ v \end{pmatrix}$ is the velocity vector, and  $p$ is the scalar pressure of a viscous fluid.

For discretization, we consider uniform meshes with $h_{x_1}=h_{x_2}=h$  and apply the MAC scheme \cite{MR1807961} to equations    \eqref{eq:Stokes1} and \eqref{eq:Stokes2}.The discrete unknowns $u, v, p$, are defined at different positions on the grid:  the velocity components $u, v$ and pressure $p$ are defined at  the middle points of vertical edges ($\Box$),  at the middle points of horizontal edges ($\circ$), and  in the center of each cell ($\smallstar$), respectively, shown in Figure \ref{fig:MAC-stokes}.  
 
 \begin{figure}[htp]
\centering
\includegraphics[width=0.5\textwidth]{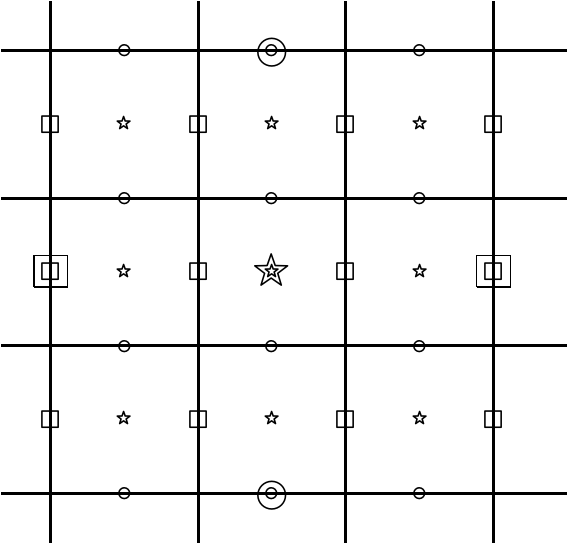}
 \caption{The staggered location of unknowns: $\Box-u,\,\, \circ-v,  \,\, \smallstar-p$. The corresponding larger shapes denote  unknowns on a coarse grid by coarsening three.}\label{fig:MAC-stokes}
\end{figure}
%===================================
 The  MAC scheme of the Stokes equations is represented by the stencils  \cite{MR1807961}
 \begin{equation}\label{eq:Lh-operator}
   \mathcal{L}_h   =\begin{pmatrix}
      -\Delta_{h} & 0 &  (\partial_{x_1})_{h/2}\\
      0 & -\Delta_{h} & (\partial_{x_2})_{h/2} \\
      -(\partial_{x_1})_{h/2}  & -(\partial_{x_2})_{h/2} & 0
    \end{pmatrix},
  \end{equation}
where 
\begin{equation*}
   -\Delta_{h} =\frac{1}{h^2}\begin{bmatrix}
       & -1 &  \\
      -1 & 4 & -1 \\
        & -1 &
    \end{bmatrix},\quad
     (\partial_{x_1})_{h/2} =\frac{1}{h}\begin{bmatrix}
       -1& 0 & 1 \\
    \end{bmatrix},\quad
     (\partial_{x_2})_{h/2} =\frac{1}{h}\begin{bmatrix}
       1 \\
      0 \\
      -1
    \end{bmatrix}.
\end{equation*}

Discretizations \eqref{eq:Stokes1}  and  \eqref{eq:Stokes2} lead to the following saddle-point system 
\begin{equation}\label{eq:saddle-structure}
       \mathcal{L}_{h} \boldsymbol{ x}=\begin{pmatrix}
      A & B^{T}\\
     B & 0\\
    \end{pmatrix}
        \begin{pmatrix} \boldsymbol{u}_{h} \\ p_{h}\end{pmatrix}
  =\begin{pmatrix} \boldsymbol{f}_h \\ 0 \end{pmatrix}=b_h,
 \end{equation}
where $A$ corresponds to the discretized vector Laplacian, $B$  stands for  the negative of the discrete divergence operator, $B^T$  is  the discrete gradient, and $\boldsymbol{u}_{h}=\begin{pmatrix} u_{h} \\ v_{h} \end{pmatrix}$.

Here, we are interested in  multigrid methods for solving  linear system \eqref{eq:saddle-structure}.  In multigrid,  there are two  important processes: smoothing and coarse-grid correction.   For a given  approximation $x_k$ and a smoother $\mathcal{M}_h$, an approximation to  $\mathcal{L}_{h} $,  the smoothing process or  relaxation scheme is
\begin{equation}\label{eq:general relaxation-form}
 \boldsymbol{ x}_{k+1} = \boldsymbol{ x}_k+\omega \mathcal{M}_h^{-1}(b_h- \mathcal{L}_{h}  \boldsymbol{ x}_k),
\end{equation}
where $\omega$ is a damping parameter to be determined. Then, the error-propagation operator for  relaxation scheme \eqref{eq:general relaxation-form} is 
\begin{equation*} 
  \mathcal{S}_h =I-\omega  \mathcal{M}_h^{-1}  \mathcal{L}_h.
\end{equation*}
Here, we consider   mass-based distributive weighted-Jacobi relaxation, Braess-Sarazin relaxation, and Uzawa-type
relaxation  proposed in \cite{YH2021massStokes} and employ  LFA to  investigate these mass-based relaxation schemes for coarsening by three in next section.
%====================================================================================================
%=====================================================================================================

\section{Local Fourier analysis}\label{sec:smoothing-analysis} 
LFA or local mode analysis \cite{MR1807961} was first introduced by Brandt \cite{brandt1977multi} to study  smoothing of multigrid methods for boundary value problems. It has since been  extended to many other problems.  LFA is used to quantitatively analyze and predict the convergence speed of multigrid methods.  There are two important factors in LFA:   smoothing factor and    two-grid convergence factor.  In many cases, the smoothing factor of LFA, assuming  an  ideal coarse grid operator that   annihilates the low frequency error components and leaves the high frequency components unchanged,   gives a sharp prediction  of actual multigrid convergence.  The two-grid convergence factor of LFA estimates the effect of  the real coarse grid operator and relaxation scheme, which sometimes poses  challenges to derive analytically  optimal  parameters in multigrid.   Thus, we focus on smoothing analysis. 

We employ LFA to study the smoothing property of  three mass-based block-structured multigrid relaxation schemes proposed in \cite{YH2021massStokes} for staggered discretizations with  coarsening by three.  High and low frequencies for  coarsening by three are defined as
\begin{equation}\label{eq:theta-domain}
  \boldsymbol{\theta}\in T^{{\rm L}} =\left[-\frac{\pi}{3},\frac{\pi}{3}\right)^{2}, \, \boldsymbol{\theta}\in T^{{\rm H}} =\displaystyle \left[-\frac{\pi}{2},\frac{3\pi}{2}\right)^{2} \bigg\backslash \left[-\frac{\pi}{3},\frac{\pi}{3}\right)^{2}.
\end{equation}
In the literature, there are some studies on LFA for coarsening by three.  For example,  an  LFA  \cite{gaspar2009geometric} was proposed  to  design  efficient geometric multigrid methods on hierarchical triangular grids using a three-coarsening strategy for a three-color and block-line type smoothers for the discrete Laplace operator.  In \cite{brannick2015local}, LFA was used  to determine automatically the optimal values for the parameters involved in defining the polynomial smoothers  for  multigrid methods with aggressive coarsening to study Poisson equation.   

Let us introduce the definition of symbol for a discrete operator \cite{MR1807961,he2018local,gaspar2009geometric} in LFA.   
\begin{definition}
Let $L_h =[s_{\boldsymbol{\kappa}}]_{h}$  be a scalar stencil operator acting on grid ${G}_{h}$ as
\begin{equation*}
  L_{h}w_{h}(\boldsymbol{x})=\sum_{\boldsymbol{\kappa}\in{V}}s_{\boldsymbol{\kappa}}w_{h}(\boldsymbol{x}+\boldsymbol{\kappa}h),
\end{equation*}
where  $s_{\boldsymbol{\kappa}}\in \mathbb{R}$ is constant,   $w_{h}(\boldsymbol{x}) \in l^{2} ({G}_{h})$, and  ${V}$ is a finite index set. 
Then, the  symbol of $L_{h}$ is defined as:
\begin{equation}\label{eq:symbol-calculation-form}
 \widetilde{L}_{h}(\boldsymbol{\theta})=\displaystyle\sum_{\boldsymbol{\kappa}\in{V}}s_{\boldsymbol{\kappa}}e^{i \boldsymbol{\theta}\cdot\boldsymbol{\kappa}},\,\, i^2=-1. 
\end{equation} 
\end{definition}
\begin{definition} The LFA smoothing factor for   the error-propagation operator $\mathcal{S}_h$  is defined as
\begin{equation}\label{eq:mu-p-form}
  \mu_{{\rm loc}}(\mathcal{S}_h(\boldsymbol{p}))=\max_{\boldsymbol{\theta}\in T^{{
  \rm H}}}\left\{ \rho(\widetilde{\mathcal{S}}_h(\boldsymbol{\theta},\boldsymbol{p}))  \right\},
\end{equation}
where  $\boldsymbol{p}$ is algorithmic parameters, such as a damping parameter or some parameters involved in smoother $\mathcal{M}_h$,  and $\rho\big(\widetilde{\mathcal{S}}_h(\boldsymbol{\theta},\boldsymbol{p})\big)$ denotes the  spectral radius of symbol $\widetilde{\mathcal{S}}_h(\boldsymbol{\theta},\boldsymbol{p})$. 
\end{definition}
For the MAC scheme considered here, $\widetilde{\mathcal{S}}_h$ is a $3\times 3$ matrix,  due to the block structure of $\mathcal{L}_h$, see \eqref{eq:Lh-operator}.   We often need to  minimize  $\mu_{{\rm loc}}(\mathcal{S}_h(\boldsymbol{p}))$  over algorithmic parameters $\boldsymbol{p}$   to obtain fast convergence speed.  We define the LFA optimal smoothing factor as follows.
\begin{definition} 
The LFA optimal smoothing factor  for  the error-propagation operator $\mathcal{S}_h$ is defined as
  \begin{equation}\label{eq:def-opt}
    \mu_{{\rm opt}}=\min_{\boldsymbol{p} \in \mathcal{R}^{+}}{\mu_{{\rm loc}}}(\mathcal{S}_h(\boldsymbol{p})).
  \end{equation}
\end{definition}
Our main goal of this work is to solve \eqref{eq:def-opt} analytically  to identify   optimal  parameters to obtain  optimal smoothing factors for  three  mass-based block-structured relaxation schemes introduced in the following. 
 
In a two-grid method, the two-grid error operator can be expressed as 
\begin{equation*}
E_h(\nu_1,\nu_2)= \mathcal{S}_h^{\nu_2} (I- P_h(L^*_H)^{-1} R_h \mathcal{L}_h)\mathcal{S}_h^{\nu_1},
\end{equation*}
where integers $\nu_1$ and $\nu_2$ are the number of pre- and postsmoothing steps,  respectively.   $L^*_H$ is the coarse-grid operator, and $P_h$ and $R_h$ are interpolation and restriction operators, respectively.  The choices of $L^*_H, R_h,$ and $P_h$ are very important for designing a good coarse-grid correction.   As to  the discrete operators on the coarser grids in the hierarchy, direct discretization of the continuous operators is used.

 For the interpolation and restriction  operators, we will consider the following choices,  \eqref{eq:P-2D-stencil-25} and \eqref{eq:R-2D-stencil}, and apply LFA to study  the corresponding two-grid convergence factors.  

\begin{equation}\label{eq:P-2D-stencil-25}
 P_{h,25}   =  \frac{1}{9}
 \left ]   \begin{tabular}{ccccc}
   1 &  2   & 3&  2 &1\\
   2&  4   & 6&  4 &2\\
  3 &  6   & 9&  6 &3\\
  2&  4   & 6&  4 &2\\
  1 &  2   & 3&  2 &1
  \end{tabular}
\right [,
\end{equation}
and the restriction is  taken to  be 

 \begin{equation}\label{eq:R-2D-stencil}
 R_{h,1}   =   \begin{bmatrix} 1
   \end{bmatrix}, \quad  
 R_{h,9}   =  \frac{1}{9}\begin{bmatrix}
   1 &  1   & 1  \\
   1 &  1   & 1  \\
  1 &  1   & 1  
   \end{bmatrix}, \quad  
 R_{h,9b}   =  \frac{1}{16}\begin{bmatrix}
   1 &  2   & 1  \\
   2 &  4   & 2  \\
  1 &  2  & 1  
   \end{bmatrix}.
\end{equation}

\begin{remark}
Some adaption is needed to compute the  symbols of $R_h$ and $P_h$ compared with these defined on collected grids (all discrete unknowns are defined at the same location). \cite[Section 3.4] {farrell2021local} gives a general formula to compute symbols of grid-transfer operators defined on general grids, which can be used here.
\end{remark}

\begin{definition}
Let $\widetilde{\mathbb{E}}_h$ be the two-grid symbol of $E_h$. Then, two-grid LFA convergence factor of $E_h$ is defined as
\begin{equation}\label{eq:rho-form}
\rho_h(\nu_1,\nu_2)  =\max_{\boldsymbol{\theta}\in T^{{\rm L}}}\left\{ \rho(\widetilde{\mathbb{E}}_h(\boldsymbol{\theta},\boldsymbol{p})) \right\},
\end{equation} 
where  $\rho(\widetilde{\mathbb{E}}_h(\boldsymbol{\theta},\boldsymbol{p}))$ denotes the spectral radius of matrix  $\widetilde{\mathbb{E}}_h(\boldsymbol{\theta},\boldsymbol{p})$.
\end{definition}

In the two-grid method, coarse and fine grid operators are involved. The dimension of the space of $3h$-harmonics for  $\boldsymbol{\theta}$  is $3\times 3=9$, see \cite{gaspar2009geometric}.  The discrete Stokes equations is a $3\times 3$ block system. Finally,  $\widetilde{\mathbb{E}}_h(\boldsymbol{\theta},\boldsymbol{p})$ is a $27\times 27$ matrix.  Note that for fixed $\nu_1$ and $\nu_2$,  $\rho_h$ in  \eqref{eq:rho-form} is a function of $\boldsymbol{p}$. In general, it is hard to theoretically minimize $\rho_h$ over $\boldsymbol{p}$.  Thus, our focus will be on solving \eqref{eq:def-opt}.  Then, we use the corresponding  optimal parameters $\boldsymbol{p}$ to compute the two-grid LFA convergence factor $\rho_h$, and compare with actual multigrid performance.    For simplicity, throughout the rest of this paper,   we drop the subscript
$h$, except when necessary for clarity.

%==============================================================================================
  \subsection{Mass matrix approximation to scalar Laplacian}
In order to design good smoothers for \eqref{eq:saddle-structure},  the approximation to  the discrete Laplacian $A$ plays an important role.  Our recent work \cite{CH2021addVanka} shows that   mass matrix obtained from bilinear elements in two dimensions is a good  smoother for  solving discrete Laplacian with standard coarsening.  Here, we analyze  ``mass approximation'' in multigrid with coarsening by three. The mass stencil for bilinear discretization  in 2D is given by
\begin{equation}\label{eq:mass-2D-stencil}
  Q    =  \frac{h^2}{36}\begin{bmatrix}
   1 &  4    & 1\\
   4 &  16   & 4\\
   1 &  4    & 1
   \end{bmatrix}.
\end{equation}
Using \eqref{eq:symbol-calculation-form}, the symbol of $Q$ is
\begin{equation}\label{eq:symbol-mass-Q}
  \widetilde{Q}(\theta_1,\theta_2)  =\frac{h^2}{9}(4+2\cos\theta_1+2\cos\theta_2+\cos\theta_1\cos\theta_2).
\end{equation}
%==================================================================================
Recall  that the standard five-point scheme for  $-\Delta_{h}=A_s$ in \eqref{eq:Lh-operator}, whose symbol is
\begin{equation}\label{eq:scalar-symbol-Laplace}
\widetilde{A}_s  =\frac{4-2\cos\theta_{1}-2\cos\theta_2}{h^2}.
\end{equation}
From \eqref{eq:symbol-mass-Q} and \eqref{eq:scalar-symbol-Laplace}, we have
\begin{equation} \label{eq:AM-MA-symbol}
\widetilde{Q} \widetilde{A}_s=\frac{2}{9}(4+2\cos\theta_1+2\cos\theta_2+\cos\theta_1\cos\theta_2) (2-\cos\theta_{1}-\cos\theta_2) .
\end{equation}
%====================================================
For mass-based relaxation applied  to the scalar Laplacian with coarsening by three, we now present the optimal smoothing factor.
\begin{theorem}\label{thm:mass-Laplace-opt-mu}
Consider the relaxation error operator $S_s= I-\omega  QA_s $. For $\boldsymbol{\theta} \in T^{\rm H}$, we have
\begin{equation}\label{eq:QA-min-max-values}
 (\widetilde{Q} \widetilde{A}_s)_{\rm min} =\frac{5}{6}, \quad  (\widetilde{Q} \widetilde{A}_s)_{\rm max} =\frac{16}{9}.
\end{equation} 
Moreover, the smoothing factor for $S_s$ with coarsening by three is 
\begin{equation*}
\mu_{\rm opt}(S_s) = \min_{\omega} \max_{\boldsymbol{\theta} \in T^{\rm H}} \{|1-\omega  \widetilde{Q} \widetilde{A}_s|\} =\frac{17}{47} \approx 0.362,
\end{equation*} 
provided that $\omega=\frac{36}{47}$.
\end{theorem}

Before giving the proof, we comment that the optimal smoothing for  $S_s= I-\omega  QA_s $ with standard coarsening is $\frac{1}{3}$, as shown in \cite{CH2021addVanka}.  From Theorem \ref{thm:mass-Laplace-opt-mu}, we see that  the optimal  smoothing factor of  $\frac{17}{47}\approx 0.362$   for coarsening by three  is very close to   $\frac{1}{3}=0.333$.
\begin{proof}
Let $x=\cos \theta_1$, $y=\cos \theta_2$.
Then,  formula \eqref{eq:AM-MA-symbol} can be rewritten as $\widetilde{Q} \widetilde{A}_s=\frac{2}{9}g(x,y)$, where 
\begin{equation}\label{eq:QA-symbol}
g(x,y) = (2-x-y)(4+2x+2y+xy).
\end{equation} 
For $\boldsymbol{\theta}\in T^{\rm H}$,  see \eqref{eq:theta-domain},  it is easy to show that $(x,y) \in [-1,1]\times [-1,\frac{1}{2}] \bigcup [-1,\frac{1}{2}]\times [\frac{1}{2},1]=:\mathcal{D}$.  To find   minimum and maximum of $\widetilde{Q} \widetilde{A}_s$ with $\boldsymbol{\theta} \in T^{\rm H}$, we start by computing the partial derivatives of $g(x,y)$:
 \begin{align*}
 g'_x &=  -(2+y)(2x+y),\\
 g'_y  &=  -(2+x)(2y+x).
 \end{align*}
Let $g'_x=g'_y=0$ with $\boldsymbol{\theta} \in T^{\rm H}$. We have $x=y=0$, that is, $(\cos\theta_1,\cos\theta_2)=(0,0)$, and  $g(0,0)=8$ which might be an extreme value.

Next, we only to find the extreme values of $g(x,y)$ at the boundary of $\mathcal{D}$. However, due to the symmetry of \eqref{eq:QA-symbol},  we only need to consider the following three boundaries:  
\begin{equation*}
\partial \mathcal{D}_1 =\{-1\} \times [-1,1], \, \partial \mathcal{D}_2=\{1\} \times \left[-1,\frac{1}{2}\right], \,
\partial \mathcal{D}_3 =\left [\frac{1}{2},1\right]\times \left\{\frac{1}{2}\right\}. 
\end{equation*} 

\begin{enumerate}
\item For $ (x,y)\in \partial \mathcal{D}_1$, 
\begin{equation*}
g(x,y) =g(-1,y) = (3-y)(2+y).
\end{equation*} 
For  $y \in[-1,1]$, we have  
\begin{align*}
g(-1,y) _{\rm max} &= g(-1,1/2)  = \frac{25}{4},\\ 
g(-1,y) _{\rm min} &= g(-1,-1)  =4.
\end{align*}

\item For $ (x,y)\in \partial \mathcal{D}_2$, 
\begin{equation*}
g(x,y) =g(1,y) = 3(1-y)(2+y).
\end{equation*} 
For  $y \in[-1,1/2]$, we have  
\begin{align*}
g(1,y) _{\rm max} &= g(1,-1/2)  = \frac{27}{4},\\
g(1,y) _{\rm min} &= g(1,1/2)  =\frac{15}{4}.
\end{align*}

\item For $ (x,y)\in \partial \mathcal{D}_3$, 
\begin{equation*}
g(x,y) =g(x,1/2) = \frac{5}{2}(3/2-x)(2+x).
\end{equation*} 
For  $x \in[1/2,1]$, we have  
\begin{align*}
g(x,1/2) _{\rm max} &= g(1/2,1/2)  = \frac{25}{4},\\
g(x,1/2) _{\rm min} &= g(1,1/2)  =\frac{15}{4}.
\end{align*}
\end{enumerate}
Based on the above discussions, when  $\boldsymbol{\theta}\in T^{\rm H}$, $g(x,y)_{\rm max} =g(0,0)=8$ and $g(x,y)_{\rm min}=g(1,1/2)=\frac{15}{4}$. It follows that when  $\boldsymbol{\theta}\in T^{\rm H}$,
\begin{equation*} 
 (\widetilde{Q} \widetilde{A}_s)_{\rm max} =\frac{2}{9}\times 8=\frac{16}{9}, \quad   (\widetilde{Q} \widetilde{A}_s)_{\rm min} =\frac{2}{9}\times\frac{15}{4}=\frac{5}{6}.
\end{equation*} 
Thus,
\begin{equation*}
\max_{\boldsymbol{\theta} \in T^{\rm {H}}} |1-\omega \widetilde{Q} \widetilde{A}_s|=\max \left\{\left|1- \frac{16}{9}\omega\right|, \left|1- \frac{5}{6}\omega\right| \right\}.
\end{equation*} 
To minimize $\max_{\boldsymbol{\theta} \in T^{\rm {H}}} |1-\omega \widetilde{Q} \widetilde{A}_s|$ over $\omega$, it requires that
\begin{equation*}
\left|1- \frac{16}{9}\omega \right| =\left|1- \frac{5}{6}\omega\right|, 
\end{equation*} 
which gives $\omega=\frac{2}{16/9+5/6}=\frac{36}{47}$. Furthermore, 
\begin{equation*}
\mu_{\rm opt}=1-\frac{5}{6}\times \frac{36}{47}=\frac{17}{47},
\end{equation*}
completing the proof.
 \end{proof}
Based on mass matrix, $Q$, we  proposed  in \cite{YH2021massStokes} three mass-based block-structured multigrid relaxation schemes with standard coarsening for the Stokes equations. Here, we  consider coarsening by three with the same mass-based block-structured relaxation schemes for the Stokes equations. We are interested in deriving optimal smoothing factors for these relaxation schemes with coarsening by three.  Since  the Laplacian  appears in Stokes equations,  Theorem \ref{thm:mass-Laplace-opt-mu} is very useful to carry out our analysis of   mass-based block-structured multigrid relaxation schemes for the Stokes equations in the next subsections. 
 
For simplicity, let  $m_s=(\widetilde{Q}/h^2)^{-1}$ and $m(\boldsymbol{\theta})=\sin^{2}(\theta_1/2)+\sin^{2}(\theta_2/2)$. Then, from  \eqref{eq:AM-MA-symbol} and \eqref{eq:QA-min-max-values}, we have  
 \begin{equation}\label{eq:mr-range-high}
 \widetilde{Q}\widetilde{A}_s=\frac{4m}{m_s}=:m_r \in \left [\frac{5}{6}, \frac{16}{9}\right]\quad  \text{for}\,\, \boldsymbol{\theta} \in T^{\rm H}.
 \end{equation}

It can be shown that the symbol of $\mathcal{L}$ defined in \eqref{eq:Lh-operator}  is 
\begin{equation*} 
   \widetilde{\mathcal{L}}(\theta_1,\theta_2) =\frac{1}{h^2}\begin{pmatrix}
      4 m  & 0 & i 2h \sin\frac{\theta_1}{2}  \\
      0 &  4m  & i 2h \sin\frac{\theta_2}{2} \\
      -i 2h \sin\frac{\theta_1}{2}   &  -i 2h \sin\frac{\theta_2}{2} &0
    \end{pmatrix}.
\end{equation*}

\subsection{Mass-based distributive relaxation}
 We review distributive relaxation following \cite{MR558216,oosterlee2006multigrid}. To relax  $\mathcal{L} \boldsymbol{ x} = b$, we  consider a transformed system $\mathcal{L}  \mathcal{P}  \hat{ \boldsymbol{ x}}= b$, where $ \mathcal{P} \hat{ \boldsymbol{ x}}= \boldsymbol{ x}$. Then, we can solve the new system with  coefficient matrix  $\mathcal{L}  \mathcal{P}$ efficiently.  Here, $ \mathcal{P} $ is given by
 \begin{equation*}
 \mathcal{P} =\begin{pmatrix}
      I_{h} & 0 &  (\partial_{x_1})_{h/2}\\
      0 & I_{h} & (\partial_{x_2})_{h/2} \\
      0  &0  & \Delta_{h}
    \end{pmatrix}.
  \end{equation*}
The  discrete matrix form  of  $\mathcal{P}$ is
\begin{equation*}
  \mathcal{P} =  \begin{pmatrix}
      I & B^{T}\\
      0 & -A_p\\
    \end{pmatrix},
\end{equation*}
where $-A_p$ is the standard five-point stencil of the Laplacian operator discretized at cell centers.  It follows that
\begin{equation}\label{DWJ-system}
      \mathcal{K}=\mathcal{L} \mathcal{P} =\begin{pmatrix}
      -\Delta_{h} & 0 &  0\\
      0 & -\Delta_{h} & 0 \\
      -(\partial_{x_1})_{h/2}  &-(\partial_{x_2})_{h/2}  & -\Delta_{h}
    \end{pmatrix}.
  \end{equation}
 %============================================================================
We can apply  block relaxation to the transformed system  $\mathcal{K} \hat{ \boldsymbol{ x}}= b$, for example, distributive weighted-Jacobi  relaxation \cite{he2018local} and distributive Gauss-Seidel relaxation \cite{MR558216,oosterlee2006multigrid}.  The former is simple, but has a convergence factor of $0.6$ \cite{he2018local}.   The latter is not suitable for parallel computation although it  has a two-grid convergence factor of 0.4 shown in \cite{he2018local}.    Here, we consider mass-based distributive relaxation (refer to $Q$-DR)  proposed in \cite{YH2021massStokes}, where $\mathcal{K}$ is approximated by $\mathcal{M}_D$  given by 

\begin{equation} \label{eq:DWJ-Precondtioner}
   \mathcal{M}_D =  \begin{pmatrix}
       \alpha_D C & 0\\
      B &   \alpha_D E \\
    \end{pmatrix}
\end{equation}
with
\begin{equation}\label{eq:C-inverse-Q}
C^{-1}  = 
\begin{pmatrix}
Q &  0\\
0 & Q
\end{pmatrix}
\end{equation}
 and $E=Q^{-1}_p$, where $Q_p$ is  the discrete mass matrix for the pressure unknowns.

It can be shown that the error propagation operator for the distributive relaxation scheme is given by  $\mathcal{S}_D=I-\omega_D \mathcal{P}\ \mathcal{M}^{-1}_D\mathcal{L}$.   Next, we examine the smoothing factor for $\mathcal{S}_D$.

  The symbol of operator $\mathcal{K}=\mathcal{L}\mathcal{P}$, see \eqref{DWJ-system},  is given by
\begin{equation*} 
  \widetilde{\mathcal{K}}(\theta_1,\theta_2) =\frac{1}{h^2}\begin{pmatrix}
       4m(\boldsymbol{\theta})& 0 & 0  \\
      0 &  4m(\boldsymbol{\theta})& 0 \\
      -i 2h \sin\frac{\theta_1}{2}   &  -i 2h \sin\frac{\theta_2}{2} & 4m(\boldsymbol{\theta})
    \end{pmatrix}.
\end{equation*}
The symbol of the block relaxation operator \eqref{eq:DWJ-Precondtioner} with $C^{-1}$ given by the mass approximation \eqref{eq:C-inverse-Q} is
  \begin{equation*} 
   \widetilde{\mathcal{M}}_{D}(\theta_1,\theta_2) =\frac{1}{h^2}\begin{pmatrix}
        \alpha_D m_s(\boldsymbol{\theta})  & 0 & 0  \\
      0 &    \alpha_D m_s(\boldsymbol{\theta}) & 0 \\
      -i 2h \sin\frac{\theta_1}{2}   & -i 2h \sin\frac{\theta_2}{2} & \alpha_D m_s(\boldsymbol{\theta})
    \end{pmatrix}.
\end{equation*}
Then, the eigenvalues of  $\mathcal{\widetilde{S}}_{D}( \alpha_D,\omega_{D},\boldsymbol{\theta})=I- \omega_{D}\widetilde{ \mathcal{P} } \widetilde {\mathcal{M}}_{D}^{-1}\widetilde{\mathcal{L}}$ are
$1-\omega_{D}\frac{4m(\boldsymbol{\theta})}{ \alpha_D m_s(\boldsymbol{\theta})}.$ Now, we are able to give the optimal smoothing factor for $Q$-DR. 
%=====================================================================
\begin{theorem}\label{thm:QDRsmoothing-theorem}
The optimal smoothing factor for $Q$-DR with coarsening by three is
  \begin{equation*} 
   \mu_{{\rm opt},D}=\min_{ \alpha_D, \omega_{D}}\max_{ \boldsymbol{\theta}\in T^{{\rm H}}}\{ \rho(\mathcal{\widetilde{S}}_{D}( \alpha_D,\omega_{D},\boldsymbol{\theta}))\}=\frac{17}{47}\approx 0.362,
\end{equation*}
where the minimum is uniquely achieved at  $\frac{\omega_{D}}{ \alpha_D}=\frac{36}{47}$.
\end{theorem}
%=======================================================================
 \begin{proof}
Since all eigenvalues of  $\mathcal{\widetilde{S}}_{D}$  are  $1-\omega_D\frac{4m }{ \alpha_D m_s }=1-\frac{\omega_D}{ \alpha_D} \widetilde{Q}\widetilde{A}_s$, using Theorem \ref{thm:mass-Laplace-opt-mu}, we know that $\mu_{\rm opt}(I- \frac{\omega_{D}}{ \alpha_D}QA_s)=\frac{17}{47}$  with $\frac{\omega_D}{ \alpha_D}=\frac{36}{47}$.  Thus, 
  $ \displaystyle \min_{\alpha_D, \omega_{D}} \max_{\boldsymbol{\theta}\in T^{{\rm H}}}\{\rho(\mathcal{\widetilde{S}}_{D}( \alpha_D,\omega_{D},\boldsymbol{\theta}))\}=\frac{17}{47}.$
 \end{proof}

In \cite{YH2021massStokes}, for the standard coarsening, we have shown that the optimal smoothing factor for $Q$-DR for MAC discretization of the Stokes equations is $\frac{1}{3}$.  From Theorem \ref{thm:QDRsmoothing-theorem}, we see that  the optimal smoothing factor for coarsening by three  is competitive with  that for  standard coarsening. 

Assume that the cost of one  cycle  standard multigrid (using one pre-smoothing and no postsmoothing step) is $W$ and the cost of coarsening by three is roughly $W/3$.  To achieve  tolerance $\epsilon$, the total  cost of standard coarsening and coarsening by three are 
\begin{equation*}
T_1=W {\rm log}_{1/3} \epsilon, \quad T_2=\frac{W}{3}{\rm log}_{17/47} \epsilon,
\end{equation*}
respectively.  Then, 
\begin{equation}\label{eq:ration-computation}
\frac{T_1}{T_2} \approx 2.78. 
\end{equation}
So coarsening by three is computationally beneficial.
 
%=======================================================================================

\subsection{Mass-based Braess-Sarazin  relaxation}
We consider mass-based Braess-Sarazin relaxation scheme proposed in \cite{YH2021massStokes}, named $Q$-BSR, where $\mathcal{M}_h$ in  \eqref{eq:general relaxation-form} is given by
\begin{equation}\label{eq:Precondtion}
   \mathcal{M}_B=  \begin{pmatrix}
      \alpha_B C & B^{T}\\
     B & 0\\
    \end{pmatrix},
\end{equation}
in which  $C^{-1}$ is defined in \eqref{eq:C-inverse-Q}, and $\alpha_B>0$ is  to be determined.  In \eqref{eq:general relaxation-form}, let $\delta \boldsymbol{ x} =\mathcal{M}_B^{-1}(b_h- \mathcal{L}_{h} \boldsymbol{ x}_k)$. Then, the update $\delta \boldsymbol{ x}=(\delta_{\boldsymbol{u}}, \delta_p)$  is given by  
\begin{eqnarray}
  (BC^{-1}B^{T})\delta_p&=&BC^{-1} r_{\boldsymbol{u}}-\alpha_B r_{p}, \label{eq:solution-of-precondtion}\\
  \delta_{\boldsymbol{u}}&=& \alpha_B^{-1}C^{-1}(r_{\boldsymbol{u}}-B^{T}\delta_p),\nonumber
\end{eqnarray}
where $(r_{\boldsymbol{u}},r_{p})=b_h-\mathcal{L}_h\boldsymbol{ x}_k$.

For the Schur complement system \eqref{eq:solution-of-precondtion}, we will consider exact solve  and inexact solve discussed.
%========================================================================================================

 \subsubsection{ Exact Braess-Sarazin relaxation} 
We first consider solving \eqref{eq:solution-of-precondtion} exactly, and  derive  optimal smoothing factor for the corresponding exact  $Q$-BSR.   
Using $m_s=(\widetilde{Q}/h^2)^{-1}$,  the symbol of $\mathcal{M}_{B}$ is  
\begin{equation*} 
   \widetilde{\mathcal{M}}_{B}(\theta_1,\theta_2) =\frac{1}{h^2}\begin{pmatrix}
       \alpha_B m_s & 0 & i 2h \sin\frac{\theta_1}{2}  \\
      0 &   \alpha_B  m_s & i 2h \sin\frac{\theta_2}{2} \\
      -i 2h \sin\frac{\theta_1}{2}   &  -i 2h \sin\frac{\theta_2}{2} &0
    \end{pmatrix}.
\end{equation*}
It can be easily shown that  the determinant of $\widetilde{\mathcal{L}}-\lambda \widetilde{\mathcal{M}}_{B}$ is 
\begin{equation*}
 4 \alpha_B  m m_s \big(\lambda-1\big)^2\left (\lambda-\frac{4 m }{\alpha_B m_s}\right).
\end{equation*}
It follows that the eigenvalues of $\widetilde {\mathcal{M}}_{B}^{-1}\mathcal{\widetilde{ L}}$ are  $1, 1$,  and $\displaystyle\frac{4m }{\alpha_B m_s}$. Note that $\displaystyle\frac{4m }{\alpha_B m_s}$ is the  eigenvalue of $\widetilde{Q}\widetilde{A}_s$, and $\frac{5}{6}<1< \frac{16}{9}$.  From Theorem \ref{thm:mass-Laplace-opt-mu}, we directly have the following result. 

%====================================================================
\begin{theorem} 
The optimal smoothing factor for exact $Q$-BSR with coarsening by three is
  \begin{equation*} 
   \mu_{{\rm opt},B}=  \min_{\alpha_B, \omega_{B}}\max_{ \boldsymbol{\theta}\in T^{{\rm H}}} \{\rho(\mathcal{\widetilde{S}}_{B}(\alpha_B, \omega_{B},\boldsymbol{\theta}))\}=\frac{17}{47},
\end{equation*}
where the minimum is uniquely achieved at $\frac{\omega_{B}}{\alpha_B}=\frac{36}{47}$ with $\omega_B \in [30/47,64/47]$.
\end{theorem}

We point out that for  standard coarsening,  the optimal smoothing factor for exact $Q$-BSR for MAC discretization of the Stokes equations is $\frac{1}{3}$  \cite{YH2021massStokes}.  Thus, using \eqref{eq:ration-computation} for exact $Q$-BSR, coarsening by three is better than standard coarsening.

%==================================================================================================

 \subsubsection{Inexact Braess-Sarazin  relaxation}
Solving the Schur complement system  \eqref{eq:solution-of-precondtion} with coefficient matrix $BC^{-1}B^{T}$ directly is expensive.  Many studies have shown that a good approximation  for \eqref{eq:solution-of-precondtion} is sufficient \cite{MR1810326}, for example, using a simple sweep of a Gauss-Seidel \cite{MR1049395} or weighted Jacobi iteration \cite{he2018local,YH2021massStokes}.   Moreover, our previous work \cite{YH2021massStokes} exhibits that for standard coarsening inexact $Q$-BSR can achieve the  same convergence factor as the exact  $Q$-BSR  for MAC discretization of the Stokes equations. Thus, we also consider mass-based  inexact Braess-Sarazin relaxation (refer to $Q$-IBSR), where we apply a single sweep of weighted ($\omega_J$) Jacobi iteration to approximate the solution of \eqref{eq:solution-of-precondtion}.  Theoretical analysis of the optimal smoothing factor is challenging due to complex eigenvalues of $\mathcal{\widetilde{S}}$, the inexact version, so we numerically study the performance of  inexact BSR under the condition that $\frac{\omega_B}{\alpha_B}=\frac{36}{47}$ because $1-\frac{\omega_B}{\alpha_B}\frac{m}{m_s}$ is an eigenvalue for both exact and inexact version. Our findings in Section \ref{sec:Numer}  show that inexact  BSR   of a two-grid method can obtain the same convergence  factor, $\frac{17}{47} $, as the exact version  with   $\omega_B=1, \alpha_B=\frac{47}{36}\omega_B$ and $ \omega_J=0.9$.

%============================================================================================
\subsection{Mass-based $\sigma$-Uzawa  relaxation}\label{Uzawa-type}

%========================================================================================================
  Uzawa-type relaxation is  a popular family of algorithms for solving saddle-point systems \cite{MR1302679,MR833993}. We consider  mass-based $\sigma$-Uzawa, called $Q$-$\sigma$-Uzawa, proposed in \cite{YH2021massStokes}, that is, $\mathcal{M}_h$ in \eqref{eq:general relaxation-form} is taken to be
\begin{equation*} 
   \mathcal{M}_{U} =  \begin{pmatrix}
      \alpha_U C & 0\\
     B & -\sigma^{-1}I\\
    \end{pmatrix},
\end{equation*}
where $  C^{-1}$ is  defined in \eqref{eq:C-inverse-Q}.

To identify the optimal smoothing factor for $\mathcal{S}_U=I-\omega_U \mathcal{M}_U^{-1}\mathcal{L}$, we first compute the eigenvalues of  $ \widetilde{\mathcal{M}}_U^{-1} \widetilde{\mathcal{L}}$. The symbol of  $\mathcal{M}_{U}$  is
\begin{equation*} 
   \widetilde{\mathcal{M}}_{U}(\theta_1,\theta_2) =\frac{1}{h^2}\begin{pmatrix}
        \alpha_{U}m_s & 0 & 0  \\
      0 &   \alpha_{U}m_s  & 0 \\
      -i 2h \sin\frac{\theta_1}{2}   &  -i 2h \sin\frac{\theta_2}{2} &-\sigma^{-1}h^2
    \end{pmatrix}.
    \end{equation*}
It can be shown that  the determinant of $\widetilde{\mathcal{L}}-\lambda \widetilde{\mathcal{M}}_{U}$ is 
\begin{equation}\label{eq:det-Uzawa}
\frac{ (\alpha_{U}m_s)^{2}}{\sigma} \left(\lambda-\frac{4m}{m_s\alpha_{U}}\right)\left(\lambda^{2}-\frac{1+\sigma}{\alpha_{U}m_s}4m\lambda+\frac{4m\sigma}{\alpha_{U}m_s}\right) .
\end{equation}
Recall that  $m_r=\frac{4m}{m_s}$, see \eqref{eq:mr-range-high}. We rewrite \eqref{eq:det-Uzawa}  as
\begin{equation}\label{eq:sigma-Uzawa-roots}
 \frac{ (\alpha_{U}m_s)^{2}}{\sigma} \left(\lambda-\frac{m_r}{ \alpha_{U}}\right)\left(\lambda^{2}-\frac{(1+\sigma)m_r }{\alpha_{U} }\lambda+\frac{m_r\sigma}{\alpha_{U} }\right).
\end{equation}
From \eqref{eq:sigma-Uzawa-roots}, we know that the eigenvalues of $\widetilde{\mathcal{M}}_{U}^{-1}\widetilde{\mathcal{L}}$ are $\lambda_{1,2}$, the two roots of
\begin{equation}\label{quadratic-function-root-U}
T(\lambda)=\lambda^{2}-\frac{(1+\sigma)m_r}{\alpha_{U}}\lambda+\frac{m_r\sigma}{\alpha_{U}},
\end{equation}
 and  $\lambda_{3}=\frac{m_r}{\alpha_{U}}$.

Our goal is to theoretically solve
\begin{equation}\label{eq:quadratic-smoothing-form}
\min_{(\alpha_U,\sigma,\omega_U)}\max_{\boldsymbol{\theta}\in T^{\rm H}}\left \{|1-\omega_U\lambda_{1,2}|, |1-\omega_U\lambda_3|\right \}.
\end{equation} 
For $\lambda_3$, from Theorem \ref{thm:QDRsmoothing-theorem} we have known  that the corresponding optimal smoothing factor is $\frac{17}{47}$. Thus, the optimal  result  of \eqref{eq:quadratic-smoothing-form} is not less than $\frac{17}{47}$. The theoretical result is given below. We note that it    is  less effective than the other two discussed previously. 

\begin{theorem} \label{thm:Uzawa-opt-mu}
The optimal smoothing factor for $Q$-$\sigma$-Uzawa relaxation with coarsening by three is   
 \begin{equation*} 
  \mu_{{\rm opt}, U}= \min _{(\alpha_{U},\omega_{U},\sigma)}\max_{ \boldsymbol{\theta}\in T^{{\rm H}}}\left\{\left|1-\omega_{U}\lambda_{3}\right|,\,|1-\omega_{U}\lambda_{1,2} |  \right\}
  =  \sqrt{\frac{17}{47}}\approx 0.601,
\end{equation*}
%
%where the minimum is attained provided that $m_2=m_{{\rm opt}}=\frac{169}{141}$ with
 with
\begin{eqnarray*}
     \frac{225}{47(16 \mu_{{\rm opt},U}-1) }\leq & \omega_{U} &\leq \frac{30}{47(1-\mu_{{\rm opt},U})}, \\
     &\alpha_{U}&=\frac{376\omega^2_U}{9(47\omega_U-15)}, \\
   &\sigma&= \frac{15}{47 \omega_{U}-15}. 
\end{eqnarray*}
 \end{theorem}
%=============================================================================================
We comment that parameters $\omega_U=1, \alpha_U=\frac{47}{36}, \sigma =\frac{15}{32}$ are in the domain of the above optimal
parameters.  Our previous work \cite{YH2021massStokes} proved  that  for standard coarsening,  the optimal smoothing factor for $Q-\sigma$-Uzawa relaxation  for MAC discretization of the Stokes equations is $\sqrt{\frac{1}{3}}$.  Note that $\sqrt{\frac{17}{47}}\approx 0.601$ is very close to $\sqrt{\frac{1}{3}}\approx 0.577$. 
Now, we have optimal smoothing factors for the mass-based Braess-Sarazin,   $\sigma$-Uzawa,   and   distributive relaxation schemes. It is clear that  mass-based Braess-Sarazin and distributive relaxation schemes outperform   $\sigma$-Uzawa relaxation in terms of smoothing factor.  Moreover, in Section \ref{sec:Numer}, we see  degradation in convergence of  actual $Q$-$\sigma$-Uzawa  multigrid method, and the proof of Theorem \ref{thm:Uzawa-opt-mu} is tedious. Thus, we  present it  in the Appendix.

%=============================================================================================
\section{Numerical experiments}\label{sec:Numer}
 In this section, we first use optimal parameters obtained from our smoothing analysis to compute two-grid LFA convergence factors $\rho(\nu_1,\nu_2)=:\rho(\nu)$ with $\nu=\nu_1+\nu_2$ defined in \eqref{eq:rho-form}.  We  also  study the influence of different  grid-transfer operators  defined in \eqref{eq:P-2D-stencil-25} and \eqref{eq:R-2D-stencil}  on two-grid methods by LFA. Finally,   multigrid methods (two-grid and $V$-cycles) performances are presented to validate our  LFA predictions. 
%=================================================================================================
 \subsection{LFA results}
We use   optimal parameters  obtained from Section \ref{sec:smoothing-analysis} that minimize the smoothing factor and $h=\frac{1}{81}$ to conduct the test.  We study the influence of different grid-transfer operators defined in \eqref{eq:P-2D-stencil-25} and \eqref{eq:R-2D-stencil} on the performance of two-grid methods. Specifically, we consider four pairs of grid-transfer operators: $(P_h, R_h)=(P_{h,25}, R_{h,1}), (P_{h,25}, R_{h,9}), (P_{h,25}, R_{h,9b}), (P_{h,25}, P^T_{h,25}/9)$. For all two-grid LFA prediction tests,  we use  $\alpha_D=1, \omega_{D}=\frac{36}{47}$ for $Q$-DR,  $\alpha_B=1, \omega_B=\frac{36}{47}$ for $Q$-BSR, and   $\omega_U=1, \alpha_U=\frac{47}{36},\sigma=\frac{15}{32}$  for $Q$-$\sigma$-Uzawa.

In Table \ref{tab:LFA-results-PR-bilinear-injection},  we report two-grid LFA convergence factors using grid-transfer operators  $(P_{h,25}, R_{h,1})$ for three relaxation schemes.  We see that there is a degradation of $\rho_h(\nu)$ compared with  $\mu^{\nu}_{\rm opt}$, especially for $\nu=1$.  This degradation can be mitigated by using more points for the restriction operators, which we observe  in Tables \ref{tab:LFA-results-PR-bilinear-injection9}, \ref{tab:LFA-results-PR-bilinear-9B} and \ref{tab:LFA-results-PR-bilinear-bilinear}, where we use  $(P_h,R_h)=(P_{h,25}, R_{h,9}), (P_{h,25}, R_{h,9b})$, $(P_{h,25}, P^T_{h,25}/9)$, respectively. Overall,  using $(P_h, R_h)=(P_{h,25}, P^T_{h,25}/9)$ gives better results compared with other choices.  From LFA predictions, we see that $Q$-BSR  outperforms the other two relaxation schemes when $\nu$ increases.   The results also suggest that it is important to select  proper grid-transfer operators to design fast multigrid methods, and LFA is helpful to  identify good grid-transfer operators before we do actual numerical tests.  
\begin{table}[H]
 \caption{Two-grid LFA convergence factor, $\rho_h(\nu)$, using $(P_h, R_h)=(P_{h,25}, R_{h,1})$.}
\centering
\begin{tabular}{lccccc}
\hline
Method                       & $\mu_{\rm opt}$        &$\rho_h(1)$      &$\rho_h(2)$     &$\rho_h(3)$    & $\rho_h(4)$  \\ \hline
 
$Q$-DR                          &0.362                          &0.546               &0.222              &0.116      &0.087    \\
$Q$-BSR                        &0.362                           &0.515             &0.245               &0.181      &0.098  \\
$Q$-$\sigma$-Uzawa      &0.601                           &0.642             &0.377                &0.226    &0.165   \\ \hline
\end{tabular}\label{tab:LFA-results-PR-bilinear-injection}
\end{table}

\begin{table}[H]
 \caption{Two-grid LFA convergence factor, $\rho_h(\nu)$, using $(P_h, R_h)=(P_{h,25}, R_{h,9})$.}
\centering
\begin{tabular}{lccccc}
\hline
Method                       & $\mu_{\rm opt}$        &$\rho_h(1)$      &$\rho_h(2)$     &$\rho_h(3)$    & $\rho_h(4)$  \\ \hline
 
$Q$-DR                          &0.362                          &0.419               &0.205              &0.157       &0.126     \\
$Q$-BSR                        &0.362                           &0.361              &0.166               &0.097        &0.073    \\
$Q$-$\sigma$-Uzawa      &0.601                           &0.601              &0.361               &0.217        &0.154   \\ \hline
\end{tabular}\label{tab:LFA-results-PR-bilinear-injection9}
\end{table}

\begin{table}[H]
 \caption{Two-grid LFA convergence factor, $\rho_h(\nu)$,  using $(P_h, R_h)=(P_{h,25}, R_{h,9b})$.}
\centering
\begin{tabular}{lccccc}
\hline
Method                       & $\mu_{\rm opt}$        &$\rho_h(1)$      &$\rho_h(2)$     &$\rho_h(3)$    & $\rho_h(4)$  \\ \hline
 
$Q$-DR                          &0.362                          &0.431               & 0.192               &0.144       &0.117     \\
$Q$-BSR                        &0.362                           &0.361              &0.149                 &0.091        &0.052   \\
$Q$-$\sigma$-Uzawa      &0.601                           &0.601             &0.361                 &0.217        &0.150    \\ \hline
\end{tabular}\label{tab:LFA-results-PR-bilinear-9B}
\end{table} 

%htheta = 2/3*pi/64;
\begin{table}[H]
 \caption{Two-grid LFA convergence factor, $\rho_h(\nu)$, using $(P_h, R_h)=(P_{h,25}, P^T_{h,25}/9)$.}
\centering
\begin{tabular}{lccccc}
\hline
Method                       & $\mu_{\rm opt}$        &$\rho_h(1)$      &$\rho_h(2)$     &$\rho_h(3)$    & $\rho_h(4)$  \\ \hline
 
$Q$-DR                          &0.362                          &0.387               &0.257              &0.197      &0.160   \\
$Q$-BSR                        &0.362                           &0.361              &0.161               &0.123     &0.099   \\
$Q$-$\sigma$-Uzawa      &0.601                           &0.601              &0.361               &0.240     &0.197   \\
 \hline
\end{tabular}\label{tab:LFA-results-PR-bilinear-bilinear}
\end{table}

\subsection{Multigrid performance}
We consider  model problems   \eqref{eq:Stokes1} and \eqref{eq:Stokes2} on a unit domain with Dirichlet boundary conditions  with zero  solution, since in \cite{YH2021massStokes},  numerical results of  $Q$-BSR, $Q$-$\sigma$-Uzawa and  $Q$-DR multigrid methods with standard coarsening for the Stokes problems with periodic boundary conditions agree with LFA predictions. 
 The introduction of implementation of MAC scheme for the Stokes systems can be found in \cite{chen2018programming}.  For grid-transfer operators, we choose the combination $(P_h, R_h)=(P_{h,25}, P^T_{h,25}/9)$ defined in  \eqref{eq:P-2D-stencil-25} and \eqref{eq:R-2D-stencil}.   The coarsest grid is  $3\times 3$.  Experimentally measured convergence factors are computed as
\begin{equation*}
\rho^{(k)}_m=\left( \frac{||r_k||}{||r_0||}\right)^{1/k},
\end{equation*}
where $r_k=b_h-\mathcal{L}_h\boldsymbol{z}_k$ is the residual  and $\boldsymbol{z}_k$ is the $k$-th multigrid iteration. In our test, we report  $\rho^{(k)}_m=:\rho_m$ with  the smallest $k$ such that $||r_k||\leq 10^{-12}$. Again, we consider $h=\frac{1}{81}$ for all tests. 

 For  $Q$-DR, Theorem \ref{thm:QDRsmoothing-theorem} shows that the optimal parameters are $\frac{\omega_D}{\alpha_D}=\frac{36}{47}$. Under this condition, we tested a range of parameter values for the multigrid methods, and found that the choice of  $\alpha_D=0.7$ and $\omega_D=\frac{36}{47}\times 0.7$ is typically best.  For $Q$-$\sigma$-Uzawa relaxation, we use $\omega_U=1, \alpha_U=\frac{47}{36}$, and $\sigma=\frac{15}{32}$. We report two-grid and V-cycle multigrid results of  $Q$-DR  and $Q$-$\sigma$-Uzawa  relaxation in Tables \ref{tab:measured-rho-QDR} and \ref{tab:measured-rho-Uzawa}, respectively.   For problems with Dirichlet boundary conditions, we notice that the smoothing of  $Q$-$\sigma$-Uzawa  relaxation and  $Q$-DR  remains unsatisfactory near the boundary, which degrades the convergence factors compared with LFA predictions.  The  degradation has been observed in other studies, see \cite{he2018local,MR1049395}, and further pre-relaxation might be needed near the boundaries.   We point out that the influence of boundaries and of boundary conditions is not taken into account for LFA. 

\begin{table}[H]
 \caption{Measured multigrid convergence factors vs. LFA predictions for $Q$-DR.}
\centering
\begin{tabular}{lccccc}
\hline
 $\nu$                                             &1      &2     &3    & 4 \\ \hline
 
LFA prediction                                  &0.387               &0.257              &0.197      &0.160   \\
Two-grid $\rho_m$                          &0.525               &0.507               &0.443     &0.394     \\
V-cycle    $\rho_m$                          &0.797               &0.715               &0.658     &0.615    \\ \hline
\end{tabular}\label{tab:measured-rho-QDR}
\end{table} 

%============================================================================
\begin{table}[H]
 \caption{Measured multigrid convergence factors vs. LFA predictions for $Q$-$\sigma$-Uzawa . }
\centering
\begin{tabular}{lccccc}
\hline
 $\nu$                                             &1      &2     &3    & 4 \\ \hline
 
LFA prediction                                &0.601              &0.361               &0.240     &0.197   \\  
Two-grid    $\rho_m$                                         &0.745              &0.602              &0.480     &0.382   \\
V-cycle   $\rho_m$                                             & 0.759              & 0.632             &0.542      & 0.442 \\ \hline
\end{tabular}\label{tab:measured-rho-Uzawa}
\end{table} 

%============================================================================
For BSR, we consider the inexact version, that is,   we apply one sweep of the weighted Jacobi iteration to  the Schur complement system,  \eqref{eq:solution-of-precondtion}.  For $Q$-IBSR, there is no degradation on the convergence factor, which is observed for $Q$-DR and $Q$-$\sigma$-Uzawa relaxation. We find that using parameters, $\omega_B=1, \alpha_B=\frac{47}{36}$ and $ \omega_J=0.9$, gives the same convergence factors as  LFA two-grid predictions of  exact $Q$-BSR, shown in Table \ref{tab:measured-rho-IBSR}.  This  indicates that $Q$-IBSR is more robust with respect to boundary conditions compared with $Q$-DR and $Q$-$\sigma$-Uzawa relaxation.

\begin{table}[H]
 \caption{Measured multigrid convergence factors for $Q$-IBSR vs. LFA predictions for $Q$-BSR.}
\centering
\begin{tabular}{lccccc}
\hline
 $\nu$                                               &1      &2     &3    & 4 \\ \hline
 
LFA prediction                                      &0.361              &0.161                &0.123     &0.099   \\
Two-grid    $\rho_m$                            &0.349               &0.163               &0.115      &0.090     \\  
V-cycle      $\rho_m$                                           &0.350               &0.183               &0.129       &0.097     \\ \hline
\end{tabular}\label{tab:measured-rho-IBSR}
\end{table}

From the results shown in Tables \ref{tab:measured-rho-QDR}, \ref{tab:measured-rho-Uzawa}, and  \ref{tab:measured-rho-IBSR}, we conclude that $Q$-IBSR outperforms  $Q$-DR and $Q$-$\sigma$-Uzawa relaxation in terms of convergence factor, and   $Q$-IBSR  is not too sensitive to boundary conditions. Therefore, we recommend $Q$-IBSR  for practical use. 

%===========================================================================================================

\section{Conclusion}
\label{sec:concl}
 
We  have considered a  staggered finite difference discretization for the Stokes equations. For this discretization,  many of multigrid studies focus on standard coarsening. In contrast, we propose highly efficient multigrid methods with coarsening by three  to solve the resulting linear system. This coarsening-by-three strategy  leads to coarsening more quickly and reducing the number of levels.  The unknowns in different levels are nested, and, thus, the construction of grid-transfer operators is simplified, avoiding different ones needed for  two components of velocity and pressure when considering standard coarsening.  It is well-known that properly selected algorithmic parameters of multigrid methods are very important to design fast algorithms.  An LFA  is presented to quantitatively analyze three block-structured mass-based  relaxation schemes and grid-transfer operators, and help choose algorithmic parameters. We  derive  LFA optimal smoothing factors for a mass-based  Braess-Sarazin relaxation, a mass-based Uzawa relaxation, and a  mass-based distributive relaxation  with this $3h$-coarsening strategy  for the Stokes equations. 

 Our theoretical results show that the optimal smoothing factors of coarsening by three are nearly equal to those obtained from standard coarsening, 
 but  the computational cost per iteration is lower due to the three factor, and therefore coarsening by three is superior computationally.  Our results also show that the mass-based  Braess-Sarazin and  distributive relaxation schemes  have same optimal smoothing factor, which is smaller than that of the mass-based Uzawa relaxation.  Furthermore, we report two-grid and V-cycle multigrid performance for the Stokes equations with Dirichlet boundary conditions. We find that there is a degradation of actual convergence factors for mass-based Uzawa  and distributive relaxation schemes compared with LFA predictions. This might be due to the boundary conditions.  However, the  actual multigrid performance of mass-based Braess-Sarazin relaxation matches  the LFA two-grid convergence factor.  As a result,  the mass-based Braess-Sarazin relaxation is preferred. 
 
\section*{Appendix}\label{sec:appendix}
To prove Theorem \ref{thm:Uzawa-opt-mu}, we first  focus on  minimizing $\max_{\boldsymbol{\theta}\in T^{\rm H}} \{|1-\omega_U\lambda_{1,2}|\}$ over $\omega_U$, which gives a lower bound for \eqref{eq:quadratic-smoothing-form}.   We follow \cite{he2018local,YH2021massStokes} to accomplish our analysis for $Q-\sigma$-Uzawa in the following.   To analyze $\lambda_{1,2}$,  we compute  the discriminant of $T(\lambda)$, which  is 
\begin{equation*} 
  \Delta(m_r)=\frac{m_r(1+\sigma)^{2}}{\alpha_{U}^{2}}\left(m_r-\frac{4\alpha_{U}\sigma}{(1+\sigma)^{2}}\right).
\end{equation*}
Two roots of $\Delta(m_r)=0$ with respect to $m_r$ are $ m_{1}=0 $ and $m_{2}=\frac{4\alpha_{U}\sigma}{(1+\sigma)^{2}}$. 
From \eqref{quadratic-function-root-U}, we have 
\begin{equation}\label{eq:roots-formulation}
 \lambda_{1,2}=\frac{(1+\sigma)m_r}{2\alpha_{U}}\left(1\pm \sqrt{1-\frac{m_2}{m_r}}\right). 
\end{equation}
We see that when $m_r\geq m_2$,   $\lambda_1 $ and $\lambda_2 $  are real,  and when $m_r< m_2$,   $\lambda_1 $ and $\lambda_2 $ are complex.    Using  \eqref{quadratic-function-root-U},  we have
\begin{equation}
  \lambda_{1}+\lambda_{2}=\frac{m_r(1+\sigma)}{\alpha_{U}}, \quad
 \lambda_{1}\lambda_{2}=\frac{m_r\sigma}{\alpha_{U}}, \label{eq:U-roots-sum-times}
\end{equation} 
For complex eigenvalues $\lambda_{1,2}$,   we define  $\Psi =|1-\omega_U\lambda_1|$. Then,    $\Psi^{2}  = (1-\omega_{U}\lambda_{1})(1-\omega_{U}\lambda_{2})$.  Using \eqref{eq:U-roots-sum-times}, we can simplify $\Psi^{2}$ as  
\begin{equation*}
 \Psi^{2}(m_r)  =1-(\lambda_1+\lambda_2)\omega_{U}+\lambda_1\lambda_2\omega_{U}^{2} =1+\frac{\omega_{U}}{\alpha_{U}}(\omega_{U}\sigma-\sigma-1)m_r.  
\end{equation*}
Recall that $m_r(\boldsymbol{\theta}) \in[5/6,  16/9]$ for $\boldsymbol{\theta}\in T^{\rm H}$, see \eqref{eq:mr-range-high}.  We first derive a general result for  $ m_2 \geq m_r$,
%===============================================================
\begin{theorem}\label{thm:complex-max-value}
Let $\gamma=\min\{m_{2},16/9 \}$, where $m_{2}\geq \frac{5}{6}$. Then, the smoothing factor for  $\lambda_{1,2}$ with $m_r\in[5/6, \gamma]$ is
  \begin{equation*}\label{general-complex-smoothing}
    \mu^C=\max_{m_r\in[5/6,\gamma]} \Psi(m_r)=\sqrt{1+\frac{5\omega_U(\omega_U\sigma-\sigma-1)}{6\alpha_U}}\geq\sqrt{1-\frac{5}{6\gamma}},
  \end{equation*}
where the equality is achieved  if and only if $\frac{\omega_{U}}{\alpha_{U}}(\omega_U\sigma-\sigma-1)=-\frac{1}{\gamma}.$
\end{theorem}
%=================================================================
\begin{proof}
When $m_r\in[5/6, \gamma]$, $\Delta(\alpha_{U},\sigma)\leq 0$ and $|1-\omega_{U}\lambda_{1}|=|1-\omega_{U}\lambda_{2}|=\Psi(m_r)$.   Setting  $\Psi^2(m_r)<1$ gives  $\frac{\omega_U(\omega_U\sigma-\sigma-1)}{\alpha_U}<0$. Using $\gamma=\min\{m_2,16/9\}$ gives
\begin{equation*}
\Psi^2(\gamma)=1+\frac{\omega_{U}}{\alpha_{U}}(\omega_{U}\sigma-\sigma-1)\gamma \geq 1+\frac{\omega_{U}}{\alpha_{U}}(\omega_{U}\sigma-\sigma-1)m_2
=\left(1-\frac{2\omega_U\sigma}{1+\sigma}\right)^2 \geq0.
\end{equation*}
This means that $1+\frac{\omega_{U}}{\alpha_{U}}(\omega_{U}\sigma-\sigma-1)\gamma \geq 0$, that is, $\frac{\omega_{U}}{\alpha_{U}}(\omega_U\sigma-\sigma-1)\geq-\frac{1}{\gamma}$.
Furthermore,
\begin{equation*}
  \max_{m_r\in[5/6,\gamma]}\Psi(m_r)=\Psi(5/6)=\sqrt{1+\frac{5\omega_{U}}{6\alpha_{U}}(\omega_U\sigma-\sigma-1)}
  \geq\sqrt{1-\frac{5}{6\gamma}},
\end{equation*}
where the equality is achieved if and only if $\frac{\omega_U(\omega_U\sigma-\sigma-1)}{\alpha_U}=\frac{-1}{\gamma}$.
\end{proof}
 
%==================================================================================
Note that  $m_r(\boldsymbol{\theta}) \in[5/6,  16/9]$.  We consider the special situation that $m_2>\frac{16}{9}$ in Theorem \ref{thm:complex-max-value}.
 %================================================================================
\begin{corollary}\label{corol:m2-large}
For $m_{2}=\frac{4\alpha_{U}\sigma}{(1+\sigma)^{2}}>\frac{16}{9}$,   the optimal smoothing factor for  $Q$-$\sigma$-Uzawa relaxation  is not less than  $\frac{\sqrt{34}}{8}$.
\end{corollary}
%=================================================================================
\begin{proof}
Since $m_2>\frac{16}{9}$, from Theorem \ref{thm:complex-max-value} with $\gamma =\frac{16}{9}$, we know the smoothing factor for the complex modes $\lambda_{1,2}$ is 
  \begin{equation*} 
  \mu^C=\Psi(5/6)\geq\sqrt{1-\frac{5}{6\gamma}}=\sqrt{1-\frac{5}{6 \cdot \frac{16}{9} } }  =\frac{\sqrt{34}}{8}.
\end{equation*}
It follows that when $m_2>\frac{16}{9}$, the optimal smoothing for $\lambda_1, \lambda_2$ and $\lambda_3$ is not less than $\frac{\sqrt{34}}{8}$.
\end{proof}
%==================================================================================

Next,  we  give a general result for  $ m_r\geq  m_2$ for the real case. We consider   $m_r\in[\gamma,16/9]$ and $m_2\leq \gamma$. From   (\ref{eq:roots-formulation}),  we have
\begin{equation*}
|1-\omega_{U}\lambda_{1,2}|=
\left|1-\frac{(1+\sigma)\omega_{U}}{2\alpha_{U}}m_r\left(1\pm\sqrt{1-\frac{m_2}{m_r}}\right)\right|.
\end{equation*}
For simplicity, let $  \chi_{\pm}(m_r ) = \frac{m_r }{2}\left(1\pm\sqrt{1-\frac{m_2}{m_r}}\right).$
It is easy to see that  $\chi_{+}(m_r )$ is an increasing function over  $m_r \in[\tau,16/9]$.  As to $\chi_{-}(m_r )$, \cite{he2018local} has shown that it is a  decreasing function.  Thus, 
\begin{eqnarray}
  &\chi_{+}(m_r )_{\rm max}&=\chi_{+}(16/9)=\frac{8}{9}\left(1+\sqrt{1-\frac{9 m_2}{16}}\right)=:\chi_{1},\label{eq:R+Max}\\
  & \chi_{-}(m_r )_{\rm min}&=\chi_{-}(16/9)=\frac{8}{9}\left(1-\sqrt{1-\frac{9m_2}{16}}\right) =:\chi_{2}.\label{eq:R-Min}
\end{eqnarray}
Define
\begin{eqnarray}
 & \mu^R &=\max_{m_r \in[\gamma,16/9]}\left \{|1-\omega_U\lambda_1|, |1-\omega_U\lambda_2|\right \} \nonumber \\
  &&={\rm max}\bigg\{\left|1-\frac{(1+\sigma)\omega_U}{\alpha_U}\chi_1\right|,\left|1-\frac{(1+\sigma)\omega_U}{\alpha_U}\chi_2\right|\bigg\} \nonumber  \\
 & & =\left\{
  \begin{aligned} \label{eq:compare-R1-R2}
    &\frac{(1+\sigma)\omega_{U}}{\alpha_{U}}\chi_1-1, \quad {\rm if}\quad \frac{(1+\sigma)\omega_{U}}{\alpha_{U}}\geq \frac{9}{8}. \\
    &1-\frac{(1+\sigma)\omega_{U}}{\alpha_{U}}\chi_2, \quad {\rm if}\quad \frac{(1+\sigma)\omega_{U}}{\alpha_{U}}\leq \frac{9}{8}.
  \end{aligned}
               \right. 
\end{eqnarray}

Note that $\mu^R$ is a function of $m_2$. Now,  we can give a lower bound on the optimal smoothing factor for  $Q$-$\sigma$-Uzawa  relaxation for $m_2\leq\frac{5}{6}$.
\begin{theorem}\label{U-all-complex}
For $m_{2}=\frac{4\alpha_{U}\sigma}{(1+\sigma)^{2}}\leq\frac{5}{6}$, the optimal smoothing factor for $Q$-$\sigma$-Uzawa  relaxation is not less than $\frac{\sqrt{34}}{8}$.
\end{theorem}
\begin{proof}
When $m_2\leq\frac{5}{6}$, $\lambda_1, \lambda_2$ are all real.   From  \eqref{eq:R+Max}, \eqref{eq:R-Min}, and \eqref{eq:compare-R1-R2}, we see that $\mu^R$ is a decreasing function of  $m_2$. Thus, for $m_r \in[5/6,16/9]$,
  \begin{equation*}
   \mu^R(m_2)\geq \mu^R(5/6) = \left\{
  \begin{aligned}
    &\frac{8(1+\sigma)\omega_{U}}{9\alpha_{U}}\left(1+\frac{\sqrt{34}}{8}\right)-1, \quad {\rm if}\quad \frac{(1+\sigma)\omega_{U}}{\alpha_{U}}\geq \frac{9}{8}. \\
    &1-\frac{8(1+\sigma)\omega_{U}}{9\alpha_{U}}\left(1-\frac{\sqrt{34}}{8}\right), \quad {\rm if}\quad \frac{(1+\sigma)\omega_{U}}{\alpha_{U}}\leq \frac{9}{8}.
  \end{aligned}
               \right. \\
  \end{equation*}
To minimize $\mu^R(m_2)$ with respect to $\alpha_U,\omega_U,\sigma$,  it requires that $\frac{(1+\sigma)\omega_{U}}{\alpha_U}=\frac{9}{8}$ and $m_2=\frac{4\alpha_{U}\sigma}{(1+\sigma)^{2}}=\frac{5}{6}$. It follows that $\min_{(\alpha_U,\omega_U,\sigma)}\mu^R(m_2)= \frac{\sqrt{34}}{8}$.  Since there is another eigenvalue $\lambda_3$,  the optimal smoothing factor for $\lambda_{1,2}$ and $\lambda_3$ may be not less than $\frac{\sqrt{34}}{8}$.
\end{proof}
The above  discussions indicate that when $m_2>\frac{16}{9}$ or $m_2 \leq \frac{5}{6}$, the optimal smoothing factor is at least $\frac{\sqrt{34}}{8}$.  Next, we will show that  the global optimal smoothing factor for all choices of $m_2$  achieves when $\frac{5}{6}\leq m_2\leq \frac{16}{9}$.  
%=======================================================================================

%=========================================================================================
\begin{proof}
We first consider $m_{2}\in[5/6, 16/9]$. From previous discussions, we know that  for  $m_r \in[5/6, m_2]$,  $ \lambda_{1,2}$ are complex, and  for $m_r \in[m_2,16/9]$,  $ \lambda_{1,2}$ are real.  We consider  $\frac{(1+\sigma)\omega_U}{\alpha_U}=\frac{9}{8}$ such that the two expressions in \eqref{eq:compare-R1-R2}  are the same.  Furthermore, we have  $m_2=\frac{4\alpha_U\sigma}{(1+\sigma)^2}=\frac{4\cdot 8^2}{9^2}\frac{\omega^2_U\sigma }{\alpha_U}$.

For $m_r \in[5/6, m_2]$,  Theorem \ref{thm:complex-max-value} gives 
\begin{equation*}
\mu^C=\sqrt{1+\frac{5 \omega_U(\omega_U\sigma-\sigma-1)}{6\alpha_U}}=\sqrt{\frac{1}{16}+ \frac{5\omega^2_U\sigma}{6\alpha_U}}.
\end{equation*}

For $m_r \in[m_2,16/9]$,   using \eqref{eq:compare-R1-R2} gives
\begin{equation*}
\mu^R=\frac{(1+\sigma)\omega_{U}}{\alpha_{U}}\chi_1-1=\sqrt{1-\frac{9 m_2}{16}}=\sqrt{1-\frac{16\omega^2_U\sigma}{9\alpha_U}}.
\end{equation*}

Thus, for $ \frac{(1+\sigma)\omega_U}{\alpha_U}=\frac{9}{8}$,
\begin{equation}\label{eq:mini-mur-muc} 
 \max_{m_r\in[5/6,16/9]}\left \{|1-\omega_U\lambda_{1,2}| \right \}= \max \left \{\mu^R,\mu^C\right \} .
\end{equation} 

Note that $\mu^R$ is a decreasing function of $\frac{\omega^2_U\sigma}{\alpha_U}$ and $\mu^C$ is an increasing function of $\frac{\omega^2_U\sigma}{\alpha_U}$.  The minimum of \eqref{eq:mini-mur-muc}  is  achieved if and only if $\mu^R=\mu^C$ and is 
\begin{equation*}
\displaystyle \min _{(\alpha_{U},\omega_{U},\sigma),\frac{(1+\sigma)\omega_U}{\alpha_U}=\frac{9}{8}}
\max \left\{\mu^R,\mu^C\right\}=\sqrt{\frac{17}{47}}=:\mu_{{\rm opt}, U},
\end{equation*}
under the condition that 
\begin{equation}
  \frac{\omega^2_U\sigma}{\alpha_U}=\frac{135}{376},\quad   \frac{(1+\sigma)\omega_U}{\alpha_U}=\frac{9}{8}.\label{eq:optimal-condion-1}
\end{equation}
%Furthermore, $m_{{\rm opt}}:=m_2=\frac{4\cdot 8^2}{9^2} \frac{\omega^2_U\sigma}{\alpha_U}=\frac{160}{141}$.
 Since $\sqrt{\frac{17}{47}} < \frac{\sqrt{34}}{8}$, it means that $m_2 \in [5/6, 16/9] $ gives a smaller smoothing factor for $\lambda_{1,2}$. Next, we  prove that the optimal smoothing factor for $\lambda_{1,2}$  over all possible parameters  is achieved at $\frac{(1+\sigma)\omega_U}{\alpha_U}=\frac{9}{8}$. Then, we include $\lambda_3$.

Let  $a=\frac{(1+\sigma)\omega_U}{\alpha_U}$ and $b=\frac{\omega^2_U\sigma}{\alpha_U}$. Then, $m_2=\frac{4\alpha_U\sigma}{(1+\sigma)^2}=\frac{4b}{a^2}$.
 Assume that $\mu^C\leq\sqrt{\frac{17}{47}}$,  that is, 
 \begin{equation*}
 \sqrt{1+\frac{5 \omega_U(\omega_U\sigma-\sigma-1)}{6\alpha_U}}=\sqrt{1-\frac{5a}{6}+\frac{5b}{6}}\leq\sqrt{\frac{17}{47}},
 \end{equation*}
which gives  $b\leq a-\frac{36}{47}$.
 
Under condition $b\leq a-\frac{36}{47}$, we consider two situations of $a$:  

Case 1: If $a>\frac{9}{8}$,  using   \eqref{eq:compare-R1-R2}   gives
\begin{eqnarray*}
\mu^R&=&\frac{(1+\sigma)\omega_U}{\alpha_U} \frac{8}{9}\left(1+\sqrt{1-\frac{9m_2}{16}}\right)-1\\
  &=&\frac{8}{9}\left(a+\sqrt{a^2-9b/4}\right)-1\\
  &>& \frac{8}{9}  \sqrt{a^2-\frac{9}{4}(a-36/47)}\\
  &=& \frac{8}{9} \sqrt{(a- 9/8)^2+(81\times 17)/(64\times 47)}\\
  &>&\sqrt{\frac{17}{47}}.
\end{eqnarray*}

Case 2: If $a<\frac{9}{8}$,  using \eqref{eq:compare-R1-R2} gives
\begin{eqnarray*}
\mu^R&=&1-\frac{(1+\sigma)\omega_U}{\alpha_U} \frac{8}{9}\left(1-\sqrt{1-\frac{9 m_2}{16 }}\right) \\
  &=&1-\frac{8}{9} a+\frac{8}{9} \sqrt{a^2-9b/4} \\
  &>& \frac{8}{9}  \sqrt{a^2-\frac{9}{4} (a-36/47)}  \\
%  &=& \frac{8}{9} \sqrt{(a-9/8)^2+(81\times 17)/(64\times 47)} \\
  &>&\sqrt{\frac{17}{47}}. 
\end{eqnarray*}

It means that $a=\frac{(1+\sigma_U)\omega_U}{\alpha_U}=\frac{9}{8}$ gives  the optimal smoothing factor for $\lambda_{1,2}$, and the corresponding optimal smoothing factor is $\mu_{{\rm opt},U}=\sqrt{\frac{17}{47}}$.

Next, we consider $\lambda_{3}=\frac{m_r}{\alpha_U}$.  From Theorem \ref{thm:mass-Laplace-opt-mu}, we know that the optimal smoothing factor for $\lambda_3$ is $\frac{17}{47}$, which is less than $\sqrt{\frac{17}{47}}$.  Thus, the optimal smoothing factor for $\lambda_{1,2}$ and $\lambda_3$ is not less than $\sqrt{\frac{17}{47}}$. We will show it is $\sqrt{\frac{17}{47}}$. Recall that $m_r\in[5/6,16/9]$ for $\boldsymbol{\theta}\in T^{\rm H}$.  Let
\begin{equation*}
 \left |1-\frac{5\omega_{U}}{6\alpha_U}\right|\leq\mu_{{\rm opt},U}\quad  {\rm and }\,\quad \left |1-\frac{16 \omega_{U}}{9\alpha_U}\right|\leq\mu_{{\rm opt},U}.
\end{equation*}
The above two inequalities give 
\begin{equation}\label{eq:third-cond}
\frac{6}{5}(1-\mu_{{\rm opt},U})\frac{1}{\omega_U}\leq\frac{1}{\alpha_U}\leq\frac{9(1+\mu_{{\rm opt},U}) }{16}\frac{1}{\omega_U}.
\end{equation}
Note that  \eqref{eq:optimal-condion-1}    can be expressed as 
\begin{equation}
  \alpha_{U}= \frac{376\omega^2_U}{9(47\omega_U-15)},\quad  \sigma = \frac{15}{47 \omega_{U}-15}.\label{eq:parameter-condition-2}
\end{equation}
Using \eqref{eq:parameter-condition-2}, we can rewrite \eqref{eq:third-cond}  as
\begin{equation*} 
  \frac{225}{47(16 \mu_{{\rm opt},U}-1) }\leq \omega_{U}\leq \frac{30}{47(1-\mu_{{\rm opt},U})},
\end{equation*}
which gives the desired result.
\end{proof}
%===================================================================================================================================
\bibliographystyle{siam}
\bibliography{MAC3_ref}
\end{document}